\documentclass[]{llncs}

\usepackage{amsmath,amssymb,mathtools}
\usepackage{hyperref}

\usepackage{comment}

\newcommand{\R}{\mathbb{R}}

\newcommand{\Z}{\mathbb{Z}}

\newcommand{\N}{\mathbb{N}}
\newcommand{\X}{\mathbb{X}}

\newcommand{\B}{\mathcal{B}}

\newcommand{\INF}{{^\infty}}

\newcommand{\ID}{\mathrm{id}}

\newcommand{\orb}{\mathcal{O}}

\newcommand{\dill}{dill}
\newcommand{\Dill}{Dill}

\begin{document}

\title{Block Maps between Primitive Uniform and Pisot Substitutions\thanks{Research supported by the Academy of Finland Grant 131558}}

\author{Ville Salo \and Ilkka T\"orm\"a}
\institute{TUCS -- Turku Centre for Computer Science \\ University of Turku \\ \email{vosalo@utu.fi} \and TUCS -- Turku Centre for Computer Science \\ University of Turku \\ \email{iatorm@utu.fi}}

\maketitle

\setcounter{footnote}{0}

\begin{abstract}
In this article, we prove that for all pairs of primitive Pisot or uniform substitutions with the same dominating eigenvalue, there exists a finite set of block maps such that every block map between the corresponding subshifts is an element of this set, up to a shift.
\end{abstract}

\section{Introduction}
\label{sec:Intro}

To a substitution $\tau$ with a fixed point $x$, we can attach a subshift $X_\tau$ in a natural way: by taking the orbit closure of $x$. If the substitution is primitive, the resulting subshift is independent of the choice of $x$, and is uniformly recurrent (also called minimal), which is an interesting dynamical property. Not all uniformly recurrent subshifts arise this way, and intuitively the ones generated by substitutions have a simple self-similar structure. In this article, we study the sets of morphisms (shift-commuting continuous functions) between pairs of such subshifts, with the goal of showing that also these sets are in some sense quite simple.

Morphisms between associated subshifts of substitutions have also been studied in for example \cite{Du00}, where it was proved that all endomorphisms of such subshifts are automorphisms. In \cite{Co71}, it was shown that every endomorphism of a binary uniform primitive substitution is a shift map, possibly composed with a bit flip. In \cite{HoPa89}, this was generalized by proving that for certain pairs of primitive uniform substitutions, up to powers of the shift, there are finitely many block maps between their subshifts. In fact, \cite{HoPa89} considers not only continuous, but measurable functions, so their results are stronger; we only consider the topological case in this paper, but see Section~\ref{sec:Finite} for some related questions. There is also some work on the non-uniform case. For example, in \cite{Ol13} it was shown that the endomorphisms of Sturmian substitutions are shift maps. Also, results of a similar flavor (though not directly related) were obtained in \cite{CoKeLe08,CoDyKeLe12}, where an upper bound was obtained for the number of letters in a substitution $\rho$ whose subshift is topologically conjugate to that of a fixed substitution $\tau$.

We extend the (topological) results of \cite{HoPa89} to a more general class of substitution pairs, namely, pairs $(\tau, \rho)$ of primitive substitutions of `balanced growth' with the same dominating eigenvalue (asymptotic growth rate of words). The balanced growth property is implied by both uniformness and the Pisot property, and defined in Section~\ref{sec:BalanceAndInvertibility}. Our main theorem is the following.

\begin{theorem}
\label{thm:IntroTheorem}
Let $\tau$ and $\rho$ be primitive aperiodic substitutions of balanced growth whose associated matrices have the same dominant eigenvalue $\lambda > 1$, and let $X_\tau$ and $X_\rho$ be the associated subshifts. Then there exists a finite set $P$ of block maps from $X_\tau$ to $X_\rho$ such that if $f : X_\tau \to X_\rho$ is a block map, then $f = \sigma^k \circ g$ or $g = \sigma^k \circ f$ for some $g \in P$ and $k \in \N$.
\end{theorem}
\setcounter{theorem}{0}

As a corollary, for a subshift generated by a substitution satisfying the assumptions, we obtain that its endomorphism monoid is a finitely generated group where the subgroup of shift maps is of finite index (see Theorem~\ref{thm:FinitelyManyCA2way2}). For example, the result applies to
\begin{itemize}
\item the Thue-Morse substitution $a \mapsto ab, b \mapsto ba$,
\item the Fibonacci substitution $a \mapsto ab, b \mapsto a$, and
\item the Tribonacci substitution $a \mapsto ab, b \mapsto ac, c \mapsto a$.
\end{itemize}
Our result for the Thue-Morse substitution follows from \cite{Co71}, and for the Fibonacci substitution from \cite{Ol13}. We do not know a published reference for results about the endomorphism monoid of the Tribonacci substitution.

Our method is elementary, and easy to understand at least in the case that $\tau = \rho$ is uniform: A primitive uniform substitution has an `inverse' $\tau^{-1}$, and if a block map is conjugated by a primitive substitution in the sense of group theory -- by applying the transformation $f \mapsto \tau^{-1} \circ f \circ \tau$ -- its radius decreases unless it was small to begin with. Since there are finitely many block maps of a given radius, the process of repeated conjugation eventually enters a loop. Since conjugation is also bijective up to composition with a shift map, this means that some shift of $f$ must have a small radius. We give a detailed proof of this special case in Section~\ref{sec:UniformCase}, obtaining a simple proof of the topological result of \cite{HoPa89} mentioned above. The purpose of this section is mainly pedagogical, that is, it serves as an introduction to the general case.

The problem with this approach in the non-uniform case is that while a non-uniform primitive substitution has a kind of inverse as well, the map $\tau^{-1} \circ f \circ \tau$ need not be a block map. To solve this problem, we need a common generalization of substitutions and block maps, which is closed under composition. Since both substitutions and block maps are orbit-preserving and have a continuous cocycle, it is natural to consider this generalization. We call such maps \dill{} maps, and discuss their basic properties in Section~\ref{sec:OrbitPreservingMaps}.

We define an `in-radius' and an `out-radius' for \dill{} maps, which together play the role of their radius, in the sense that there are finitely many \dill{} maps having a particular combination of in- and out-radii. The behavior of these and other auxiliary properties of \dill{} maps is studied in Section~\ref{sec:Invariants}, and in Section~\ref{sec:BalanceAndInvertibility}, we discuss these properties for primitive Pisot substitutions. In particular, we show that primitive substitutions are invertible as \dill{} maps in a specific sense, and that Pisot substitutions have the balanced growth property. In Section~\ref{sec:Conjugation}, we show that when repeatedly conjugated by primitive substitutions of balanced growth (in the group theoretical sense, as described above), both radii of a \dill{} map decrease below a global bound depending only on the substitution, analogously to the uniform case. In Section~\ref{sec:Finite}, we prove our main theorem using this technique, and discuss possible extensions and optimality.

\section{Definitions and Preliminary Results}

Let $A$ be a finite set, called the \emph{alphabet}. We denote by $A^*$ the set of finite words over $A$, and by $A^\N$ and $A^\Z$ the sets of infinite one- and two-sided words over $A$, respectively. The length of a word $w$ is denoted $|w|$. The empty word is denoted by $\epsilon$, and $A^+ = A^* \setminus \{\epsilon\}$. The $i$'th letter of a word $w$ is denoted $w_i$, and \emph{subwords} (also known as \emph{factors}) can be extracted with $w_{[i,j]} = w_i w_{i+1} \cdots w_j$. If $v$ is a factor of $w$, we denote $v \sqsubset w$. If $X$ is a set of words, we write $v \sqsubset X$ if and only if $v \sqsubset x$ for some $x \in X$. The number of indices $i$ such that $v = w_{[i,i+|v|-1]}$ is denoted $|w|_v$.

We denote $A^\omega = A^* \cup A^\N$, and make all of $A^\omega$, $A^\N$ and $A^\Z$ metric spaces with $d(x,y) = \inf \{ 2^{-n} \;|\; x_{[-n,n]} = y_{[-n,n]} \}$ for all $x, y \in A^\Z$, and
\[ d(u,v) = \inf \{ 2^{-n} \;|\; \forall i \in [0,n] : u_i = v_i \mbox{~or both $u_i$ and $v_i$ are undefined} \} \]
for all $u, v \in A^\omega$. Thus, infinite words can be estimated with finite words and other infinite words in the obvious way, and finite words are isolated points.

In the following, $\X$ stands for both $\N$ and $\Z$. The \emph{shift map} $\sigma$ is a continuous self-map of $A^\X$, defined by $\sigma(x)_i = x_{i+1}$ for all $i \in \X$. A \emph{subshift} is a subset $X$ of $A^\X$ which is closed with respect to the metric $d$ and satisfies $\sigma(X) \subset X$ if $\X = \N$, and $\sigma(X) = X$ if $\X = \Z$. Alternatively, every subshift is defined by a set $F \subset A^*$ of \emph{forbidden words} as
\[ X = \{ x \in A^\X \;|\; \forall w \in F: w \not\sqsubset x \}. \]
We denote $\B_n(X) = \{ w \in A^n \;|\; w \sqsubset X \}$ for all $n \in \N$, and $\B(X) = \bigcup_{n \in \N} \B_n(X)$. For $x \in X$ we define $\orb(x) = \{\sigma^n(x) \;|\; n \in \X\}$, called the \emph{orbit} of $x$.

An infinite word $x$, or a subshift $X$, is \emph{uniformly recurrent} if for all $w \sqsubset x$ ($w \sqsubset X$) there exists $n \in \N$ such that $w \sqsubset v$ for all $v \sqsubset x$ ($v \sqsubset X$, respectively) with $|v| \geq n$. We say $x$ or $X$ is \emph{linearly recurrent} with constant $N \in \N$ if $n$ above can always be chosen as $N|w|$. Uniformly recurrent subshifts are also called \emph{minimal} in the literature because they are exactly the subshifts that contain no proper subshifts. This characterization in particular implies that $\overline{\orb(x)} = X$ for all $x \in X$.

Let $X \subset A^\X$ and $Y \subset B^\X$ be subshifts (of the same type). A continuous function $f : X \to Y$ is called a \emph{block map} if $f \circ \sigma = \sigma \circ f$. Alternatively, a block map is defined by a \emph{local function} $f' : \B_{r-\ell+1}(X) \to B$ and a finite \emph{neighborhood} $N = [\ell,r] \subset \X$ by $f(x)_i = f'(x_{[i+\ell, i+r]})$ for all $x \in X$. The morphisms considered in this article are the block maps. If $N = [0,r]$ in the case $\X = \N$, or $N = [-r,r]$ in the case $\X = \Z$, we say $r$ is a \emph{radius} of $f$. Since all continuous shift-commuting maps are defined by local functions, we also apply such functions to words, by abuse of notation: a radius $r \in \N$ and a corresponding local function $f' : \B_{2r + 1}(X) \to B$ are chosen implicitly, and for all $w \in \B(X)$ with $|w| \geq 2r + 1$, the image $f(w) \in B^{|w| - 2r}$ is defined by $f(w)_i = f'(w_{[i,i + 2r]})$.

A \emph{substitution} on the alphabet $A$ is a map $\tau : A^* \to A^*$ satisfying $\tau(uv) = \tau(u) \tau(v)$ for all $u, v \in A^*$. A substitution is effectively determined by the images of letters. The \emph{associated matrix} of $\tau$ is the $|A| \times |A|$ integer matrix whose coordinate $(a,b)$ contains $|\tau(a)|_b$ for all $a, b \in A$. This matrix has a \emph{dominating eigenvalue} $\lambda > 0$, greater in absolute value than any other eigenvalue, also known as the Frobenius eigenvalue. The \emph{one- and two-sided subshifts} of $\tau$ are both defined by the set of forbidden words
\[ \{ w \in A^* \;|\; \forall a \in A, n \in \N : w \not\sqsubset \tau^n(a) \}, \]
and denoted by $X_\tau \subset A^\N$ and $X_\tau^\leftrightarrow \subset A^\Z$, respectively. We say $\tau$ is
\begin{itemize}
\item \emph{non-erasing} if $|\tau(a)| > 0$ for all $a \in A$.
\item \emph{uniform} if $|\tau(a)| = |\tau(b)|$ for all $a, b \in A$,
\item \emph{injective} if $\tau(a) \neq \tau(b)$ for all $a \neq b \in A$,
\item \emph{primitive} if $a \sqsubset \tau^n(b)$ for all $a, b \in A$ and large enough $n \in \N$,
\item \emph{aperiodic} if $X_\tau$ does not contain a $\sigma$-periodic point, and
\item \emph{Pisot}, if the eigenvalues of its associated matrix are $\lambda_1, \ldots, \lambda_k$, where $|\lambda_1| > 1$ and $|\lambda_i| < 1$ for all $i > 1$.
\end{itemize}

A non-erasing substitution $\tau$ on $A$ acts on $A^\N$ by $\tau(x) = \tau(x_1) \tau(x_2) \ldots$ and on $A^\Z$ by $\tau(x) = \cdots \tau(x_{-2}) \tau(x_{-1}) . \tau(x_0) \tau(x_1) \cdots$, where $.$ denotes the origin. 
We also refer to such actions as substitutions. A \emph{two-sided fixed point} of $\tau$ is an element $x \in A^\Z$ such that $\tau(x) = x$ and $x_{-1}x_0 \sqsubset \tau^n(a)$ for some $a \in A$ and $n \in \N$. If $\tau$ is primitive and has a two-sided fixed point $x \in A^\Z$, then for all $w \in A^*$ we have $w \sqsubset x$ if and only if $w \sqsubset \tau^n(a)$ for some $a \in A$ and $n \in \N$. It is also well known that the subshift of a primitive substitution is linearly recurrent. In this case $X_\tau$ is equal to the orbit closure of any $\tau$-periodic point $x \in A^\N$, and thus infinite if and only if $\tau$ is aperiodic. Since the existence of a $\tau$-periodic point which is $\sigma$-eventually periodic is decidable \cite{HaLi86,Pa86}, so is the finiteness of $X_\tau$.

The following definitions of recognizability are direct generalizations of the ones from \cite{Mo92} and \cite{Mo96a}.

Let $\tau$ be a substitution on $A$ with an aperiodic fixed point $x \in A^\N$ in which all elements of $A$ occur. We denote $E_x(0) = 0$ and $E_x(p) = |\tau(x_{[0,p-1]})|$ for $p > 0$, and $E^x = \{E_x(p) \;|\; p \geq 0\}$. Let $w \in A^*$, and let $i, j \in \N$ be such that
\[ w = x_{[i,i+|w|-1]} = x_{[j,j+|w|-1]}. \]
We say $w$ \emph{has the same $1$-cutting} at $i$ and $j$ if
\[ E^x \cap [i, i+|w|-1] + (j - i) = E^x \cap [j, j+|w|-1]. \]
We say that $w$ \emph{comes from} the word $v$ at $i$ if $v = x_{[p, q-1]}$ for the unique $p$ and $q$ such that $i \in [E_x(p),E_x(p+1)-1]$ and $i + |w| - 1 \in [E_x(q-1),E_x(q)-1]$. We say that $\tau$ is \emph{bilaterally recognizable} (or simply \emph{recognizable}) if there exists $L \in \N$ such that for all $i, j, n \in \N$ with $x_{[i-L,i+n+L-1]} = x_{[j-L,j+n+L-1]}$, we have that $w = x_{[i,i+n-1]} = x_{[j,j+n-1]}$ has the same $1$-cutting at $i$ and $j$. If in addition $w$ always comes from the same word at $i$ and $j$, we say $\tau$ is \emph{strictly recognizable}.

We note here that the definition of recognizability does not depend on the choice of the fixed point $x$ (even if the set $E^x$ does), as long as all letters of $A$ occur in it, since all fixed points that share this property have the same set of factors. We also remark that strict recognizability is not a standard notion, and it is in fact implied by primitivity.

\begin{lemma}[Theorem 2 of \cite{Mo96a}]
\label{lem:Mosse}
A primitive aperiodic substitution is strictly recognizable.
\end{lemma}

\begin{lemma}
\label{lem:RecogToCA}
Let $\tau$ have a one-sided fixed point $x \in X_\tau$ in which every letter occurs, and let $x' \in X_\tau^\leftrightarrow$ be such that $x'_i = x_i$ for all $i \geq 0$. Then $\tau$ is recognizable if and only if there exists a block map $R^\circ = R^\circ_\tau : X_\tau^\leftrightarrow \to \{0,1\}^\Z$ such that for large enough $i \in \N$, $R^\circ(x')_i = 1$ iff $i \in E^x$. It is strictly recognizable if and only if there exists a block map $R = R_\tau : X_\tau^\leftrightarrow \to (\{\#\} \cup A)^\Z$ such that for large enough $i \in \N$, $R(x')_i \neq \#$ if and only if $i \in E^x$, and then $x_i$ comes from the letter $R(x')_i$ at $i$.
\end{lemma}

\begin{proof}
Suppose $\tau$ is recognizable, let $L \in \N$ be as in the definition of recognizability, and set $L$ as the radius of $R^\circ$. Then for all $w \sqsubset x$ with $|w| = 2L+1$, either $i \in E^x$ whenever $w = x_{[i-L,i+L]}$, or $i \notin E^x$ for all such $i$. In the first case, the local rule of $R^\circ$ outputs $1$ on input $w$, and in the second, $0$. If $\tau$ is also strictly recognizable with the same $L$, then $x_i$ always comes from the same letter $a \in A$ at $i$ when $w = x_{[i-L,i+L]}$. If the local rule of $R^\circ$ outputs $0$ on input $w$, then $R$ outputs $\#$; otherwise, it outputs the above $a$.

Conversely, if $r$ is the radius of $R^\circ$ ($R$, respectively), then one can take $r + \max\{|\tau(a)| \;|\; a \in A\}$ as $L$ in the definition of recognizability (strict recognizability, respectively). \qed
\end{proof}

Another important result is Theorem 2.4 of \cite{Mo92}, which we give as a lemma below.

\begin{lemma}[Theorem 2.4 of \cite{Mo92}]
\label{lem:MossePower}
Let $\tau$ be a primitive aperiodic substitution on $A$. Then there exists $N \in \N$ such that $w^N \not\sqsubset X_\tau$ for all $w \in A^+$.
\end{lemma}

In the above case, we say $X_\tau$ \emph{has bounded powers}. We also need the following simple analytical lemma.

\begin{lemma}
\label{lem:SublinearContribution}
There exists a function $\alpha : [0, 1) \times \R \to \R$ such that for any sequence of real numbers $(x_i)_{i \in \N}$ such that $x_{i+1} \leq a x_i + b$ for all $i$, we have $x_i \leq \alpha(a, b)$ for large enough $i$.
\end{lemma}

\section{The Special Case of Injective Uniform Substitutions}
\label{sec:UniformCase}

As the special case of endomorphisms of an injective uniform primitive substitution is easy to establish, we begin with a relatively self-contained treatment of it. We use Lemma~\ref{lem:MossePower} without a proof, as the proof in \cite{Mo92} is easy to read, and the result is not essentially simpler to prove in the uniform case (see also \cite{DuHoSk99}, Section~7.1 for a nice proof through linear recurrence). As the proof of Lemma~\ref{lem:Mosse} is quite complicated, we give a simpler proof of this in the injective uniform case in Lemma~\ref{lemma:Recognizability}. This section is mainly pedagogical, and its aim is to explain our proof technique: the main result of this section, Proposition~\ref{prop:UniformCase}, follows both from the results of \cite{HoPa89}, and from those in Section~\ref{sec:Finite}, which generalize it in different directions. If the reader is only interested in the proof of the main result, this section can thus be skipped. 

For the rest of this section, fix an alphabet $A$, an injective primitive uniform aperiodic substitution $\tau$ on $A$, so that $|\tau(a)| = m$ and $\tau(a) \neq \tau(b)$ for all $a, b \in A$, and suppose $\tau$ has a fixed point $x \in X_\tau$. For example, the Thue-Morse substitution $\tau(a) = ab, \tau(b) = ba$ satisfies the assumptions. Only primitivity, aperiodicity and uniformity are essential for our proof technique, and injectivity simplifies our proof. Ensuring the existence of a fixed point on the other hand is a matter of taking a power of $\tau$.

We now restate, and prove, Lemma~\ref{lem:Mosse} for the injective uniform case, as this is substantially easier than the general primitive case, which is proved in \cite{Mo96a}. In the following, an $h \bmod m$-factor of a word $u$ is another word $v$ such that $v = u_{[i,i+|v|-1]}$ for some $i$ with $i \equiv h \bmod m$.

\begin{lemma}
\label{lemma:AConfusedShift}
Suppose that for all $h \in [1, m-1]$, there exists $w_h \in A^+$ which is a $0 \bmod m$-factor, but not an $h \bmod m$-factor, of $x$. Then $\tau$ is recognizable.
\end{lemma}

\begin{proof}
Let $w \in A^*$ be a prefix of $x$ where every $w_h$ occurs as a $0 \bmod m$-factor. Then $w$ cannot occur in $x$ as an $h \bmod m$-factor for any $h \in [1, m-1]$, for otherwise $w_h$ would also be an $h \bmod m$-factor of $x$. Since $x$ is uniformly recurrent, there exists $L \in \N$ such that every length-$L$ factor of $x$ contains an occurrence of $w$. Now, for all $i \in \N$ we have $i \in E^x = \{0, m, 2m, \ldots\}$ if and only if the word $w$ is a $0 \bmod m$-factor of $x_{[i,i+L-1]}$. This proves the recognizability of $\tau$.
\qed
\end{proof}

\begin{lemma}
\label{lemma:Recognizability}
The substitution $\tau$ is recognizable.
\end{lemma}

\begin{proof}
Suppose on the contrary that $\tau$ is not recognizable. By Lemma~\ref{lemma:AConfusedShift}, every $0 \bmod m$-factor of $x$ is then also an $h \bmod m$-factor of $x$. Let $w^{(0)} \sqsubset x$ with $|w^{(0)}| = 2$ be arbitrary. For all $i \geq 1$, define $w^{(i)} = \tau(w^{(i-1)})$. Denote by $N_i \subset [0, m^i-1]$ the set of coordinates $n$ such that $w^{(i)}_{[n,n+m^i-1]} = \tau^i(a)$ for some $a \in A$.

We now claim that $|N_i| \geq i + 1$. For $i = 0$, this is clear, so suppose that $i \geq 1$, so that by the induction hypothesis we have $|N_{i-1}| \geq i$. By the definition of $\tau$, for each $n \in N_{i-1}$, we have $(w^{(i)})_{[mn,mn+m^i-1]} = \tau^i(a)$ for some $a \in A$, hence $mn \in N_i$. Since $w^{(i)}$ is the $\tau$-image of a factor of $x$, it is a $0 \bmod m$-factor of $x$. But then it is also an $h \bmod m$-factor, and since $|w^{(i)}| = 2m^i$, some $\tau^i(a)$ is a $(-h) \bmod m$-factor of $w^{(i)}$. Thus $km-h \in N_i$ for some $k \geq 1$, and since this number is not divisible by $m$, it is not the position of one of the factors introduced previously, so $|N_i| \geq i + 1$.

Let $K \geq 2$ be arbitrary. We show that $u^K \sqsubset x$ for some $u \in A^+$, contradicting Lemma~\ref{lem:MossePower}. Let $i = (K+1) |A|-1$, and consider the word $w^{(i)}$. Denote
\[ I_a = \{ n \in [0,m^i-1] \;|\; w^{(i)}_{[n,n+m^i-1]} = \tau^i(a) \} \]
for each $a \in A$. By the above, we have $\left| \bigcup_{a \in A} I_a \right| \geq (K + 1)|A|$, implying that $|I_a| \geq K+1$ for some $a \in A$. Then there exist $n_1,n_2 \in I_a$ with $0 < k = |n_1 - n_2| \leq \frac{m^i}{K}$. But then $\tau^i(a)_{[0, m^i - k - 1]} = \tau^i(a)_{[k, m^i - 1]}$, thus $\tau^i(a)$ is periodic with period $k$, and then $\tau^i(a)_{[0, k - 1]}^K \sqsubset x$, a contradiction.
\qed
\end{proof}

In fact, the proof of Lemma~\ref{lemma:Recognizability} shows something slightly stronger: that $\tau$ is `recognizable from the right', in the sense that we only need to look at $w_{[i, i+L]}$, and not $w_{[i-L, i+L]}$, to determine whether the coordinate $i$ is in $E^x$. This property is not shared by primitive substitutions in general.

Since $\tau$ is injective, recognizability of course means that it is also strictly recognizable.\footnote{A simple proof of this implication in the general primitive case is given in \cite{Mo96a}.} Let $R : X_\tau^\leftrightarrow \to \mathcal{O}(\INF (A \#^{m-1})) \INF)$ be the block map given for $\tau$ by Lemma~\ref{lem:RecogToCA}. 
Using recognizability, it is easy to see that $X_\tau^\leftrightarrow$ is a disjoint union $\bigcup_{p = 0}^{m - 1} X^p$ where $X^0 = \tau(X_\tau^\leftrightarrow)$ and $X^p = \sigma^p(X^0)$ for all $p \neq 0$. For a configuration $y \in X_\tau^\leftrightarrow$, the unique $p \in [0,m-1]$ such that $y \in X^p$ is called the \emph{period class} of $y$.

\begin{lemma}
\label{lem:PeriodClasses}
Let $f : X_\tau^\leftrightarrow \to X_\tau^\leftrightarrow$ be a block map. Then $f(X^0) \subset X^p$ for some $p \in [0,m-1]$.
\end{lemma}

\begin{proof}
Let $y, y' \in X^0$, and let $w \sqsubset X_\tau^\leftrightarrow$ be long enough that $|R(f(w))| \geq m$ (and thus also $|R(w)| \geq m$). We have $y_{[i, i+|w|-1]} = w = y'_{[j, j+|w|-1]}$ for some $i, j \in \Z$ by uniform recurrence of $X_\tau^\leftrightarrow$. Since $|R(w)| \geq m$, the word $R(w)$ contains a letter different from $\#$, and then necessarily $j - i \equiv 0 \mod m$. We also have $f(y)_{[i+r, i+r+|f(w)|-1]} = f(w) = f(y')_{[j+r, j+r+|f(w)|-1]}$, where $r \in \N$ is the radius of $f$. Since $|R(f(w))| \geq m$, the word $R(f(w))$ also contains a letter different from $\#$, and then $f(y)$ and $f(y')$ must have the same period class. \qed
\end{proof}

We can now prove the main result of this section, the following special case of Theorem~\ref{thm:FinitelyManyCA}.


\begin{proposition}
\label{prop:UniformCase}
There exists a finite set $P$ of block maps on $X_\tau^\leftrightarrow$ such that if $f : X_\tau^\leftrightarrow \to X_\tau^\leftrightarrow$ is a block map, then $f = \sigma^k \circ f'$ for some $k \in \Z$ and $f' \in P$.
\end{proposition}

\begin{proof}
We inductively define a sequence $(f_i)_{i \in \N}$ of block maps on $X_\tau^\leftrightarrow$ as follows. Let $p \in [0,m-1]$ be such that $\sigma^p(f(X^0)) \subset X^0$ (it exists by Lemma~\ref{lem:PeriodClasses}), and define $f_0 = \sigma^p \circ f$. Let then $i \in \N$, and suppose the block map $f_i : X_\tau^\leftrightarrow \to X_\tau^\leftrightarrow$ is defined and $f_i(X^0) \subset X^0$. We will `conjugate' $f_i$ by $\tau$ to obtain $f_{i+1}$ with a (hopefully) smaller radius. Since $\tau$ is a homeomorphism from $X_\tau^\leftrightarrow$ to $X_0$ by injectivity, there exists a unique function $C_\tau(f_i) : X_\tau^\leftrightarrow \to X_\tau^\leftrightarrow$ satisfying $\tau \circ C_\tau(f_i) = f_i \circ \tau$. It is also shift-invariant and continuous, and thus a block map. We again let $f_{i+1} = \sigma^q \circ C_\tau(f_i)$ for the $q \in [0,m-1]$ such that $f_{i+1}(X^0) \subset X^0$.


We define the equivalence relation $\sim$ on block maps on $X_\tau^\leftrightarrow$ by $g \sim h$ if and only if $g = \sigma^n \circ h$ for some $n \in \Z$, and note that $f_{i+1} \sim C_\tau(f_i)$ holds for all $i \in \N$. Let $g$ and $h$ be two block maps on $X_\tau^\leftrightarrow$ such that $g(X^0) \subset X^0$ and $h(X^0) \subset X^0$, and suppose that $C_\tau(g) \sim C_\tau(h)$. We will show that $g \sim h$ holds, and for that, let $n \in \Z$ be such that $C_\tau(g) = \sigma^n \circ C_\tau(h)$, and let $y \in X_\tau^\leftrightarrow$ be arbitrary. Then
\begin{align*}
g(\tau(y)) &= \tau(C_\tau(g)(y)) = \tau(\sigma^n(C_\tau(h)(y))) \\
&= \sigma^{m n}(\tau(C_\tau(h)(y))) = \sigma^{m n}(h(\tau(y))),
\end{align*}
which implies that $g = \sigma^{m n} \circ h$, since $\tau(X_\tau^\leftrightarrow) = X^0$ and $X_\tau^\leftrightarrow = \bigcup_{i=p}^{m-1} \sigma^p(X^0)$. Thus, the function $C_\tau$ is injective up to $\sim$-equivalence.

Suppose that $f_i$ has radius $r \in \N$, and let $y \in X_\tau^\leftrightarrow$ be arbitrary. Denote $z = \tau(C_\tau(f_i)(y)) = f_i(\tau(y))$. Since $r$ is a radius for $f_i$, the word $z_{[0,m-1]}$ is determined by $\tau(y)_{[-r, m+r-1]}$, which in turn is determined by $y_{[-k, k]}$, where $k = \lceil \frac{r}{m} \rceil + 1$. But since $\tau$ is injective on letters, $z_{[0,m-1]}$ determines $C_\tau(f_i)(y)_0$, and thus $C_\tau(f_i)$ has radius at most $k$. This implies that $f_{i+1}$ has radius at most $k+m$, and Lemma~\ref{lem:SublinearContribution} then shows that there exists some number $N \in \N$, independent of $f$, such that for all large enough $i \in \N$, the block map $f_i$ has radius at most $N$. Let $P$ be the set of such block maps.

Since $P$ is a finite set, we have $f_{t + j} = f_t$ for some $t \in \N$ and $j > 0$. By the injectivity of $C_\tau$ up to $\sim$-equivalence, we then have $f_{i + j} \sim f_i$ for all $i \in \N$, implying $f \sim f_0 \sim f_{n j}$ for all $n \in \N$. But $f_{n j} \in P$ when $n$ is large enough, finishing the proof. \qed
\end{proof}

We show that this result is optimal in the sense that one can construct arbitrarily many nonequivalent block maps with arbitrarily large radii on the subshift of a primitive uniform substitution. Contrast this with the result of \cite{Co71}, stating that the endomorphisms of subshifts of binary uniform substitutions are symbol maps.

\begin{example}
\label{ex:ManyUniform}
Let $m \in \N$ and $n \geq 4$, and consider the substitution $\tau$ on the alphabet $A = \{a_i, b_i \;|\; 0 \leq i < m\}$ defined by $\tau(a_i) = b_{i+1} a_i^{n-1}$ and $\tau(b_i) = b_i a_i^{n-1}$, where the indices are taken modulo $m$. This is a uniform primitive substitution whose subshift has the symbol map $f(a_i) = a_{i+1}$, $f(b_i) = b_{i+1}$ and its powers as endomorphisms. We perform a state splitting on the letter $a_0$, obtaining the substitution $\rho$ on the alphabet $\hat A = A \cup \{c\}$ defined as
\[ \begin{array}{rlrl}
  \rho(a_i) &= \tau(a_i) \mathrm{~for~} 0 \leq i < m, & \rho(c) &= \tau(a_0), \\
  \rho(b_i) &= \tau(b_i) \mathrm{~for~} 0 < i < m, & \rho(b_0) &= b_0 c^{n-1}.
\end{array} \]
The letters $c$ and $a_0$ of $\rho$ together represent the letter $a_0$ of $\tau$, with the extra information of whether its preimage is $a_0$ or $b_0$. The subshift $X_\tau^\leftrightarrow$ is isomorphic to the two-sided subshift $X_\rho^\leftrightarrow \subset \hat A^\Z$ of $rho$ via an obvious isomorphism $\phi : X_\tau^\leftrightarrow \to X_\rho^\leftrightarrow$ (induced by the state-splitting) with neighborhood $[-n+1,0]$, such that $\phi^{-1}$ has neighborhood $\{0\}$. In particular, $f$ induces an endomorphism $\hat f = \phi \circ f \circ \phi^{-1}$ on the subshift $X_\rho^\leftrightarrow$ of $\rho$.

We claim that the endomorphism $\hat f$ cannot be defined by the contiguous neighborhood $[d,d+n-2] \subset \Z$ for any $d \in \Z$. Suppose for contradiction that this is the case. First, if $d \geq n$ or $d \leq -n+1$, then $f = \phi^{-1} \circ \hat f \circ \phi$ (and thus the identity morphism) can be defined by a contiguous neighborhood not containing $0$. Since $X_\tau^\leftrightarrow$ is not a periodic subshift, this is a contradiction. Thus we may assume $-n+1 < d < n$.

Let then $x \in X_\rho^\leftrightarrow$ be such that $x_{[0,3]} = b_{m-1} a_{m-1} a_{m-1} a_{m-1}$, and denote $y = \sigma(x)$. Then
\begin{alignat*}{6}
\rho(x)_{[0,3n-1]} &= b_{m-1} & a_{m-1}^{n-1} & b_0 & a_{m-1}^{n-1} & b_0 & a_{m-1}^{n-1}, \\
\rho(y)_{[0,3n-1]} &= b_0 & a_{m-1}^{n-1} & b_0 & a_{m-1}^{n-1} & b_0 & a_{m-1}^{n-1}.
\end{alignat*}
Due to $-n+1 < d < n$, the local rule of $\hat f$, applied at coordinate $n-1$, cannot see the coordinates $0$ or $3n$. Thus, $\hat f(\rho(x))_{n-1} = \hat f(\rho(y))_{n-1}$. However, we should have $\hat f(\rho(x))_{n-1} = a_0$ and $\hat f(\rho(y))_{n-1} = c$, a contradiction.

An analogous argument can be applied to all powers of $\hat f$, except the identity map. All in all, the subshift $X_\rho^\leftrightarrow$ has at least $m-1$ endomorphisms that are pairwise distinct even modulo powers of the shift, and cannot be defined by contiguous neighborhoods of size less than $n$. \qed
\end{example}

\section{Orbit-preserving Maps and \Dill{} Maps}
\label{sec:OrbitPreservingMaps}

We now begin the preparations for the proof of Theorem~\ref{thm:IntroTheorem} (and related results). Our proof is based on orbit-preserving maps, or maps that send orbits to orbits. There is extensive literature on orbit-equivalence of Cantor dynamical systems, that is, the existence of orbit-preserving homeomorphisms; see for example \cite{GiPuSk95,BoTo98,GiPuSk09}. Our motivation is a bit different: most of the aforementioned papers seek to characterize different kinds of orbit equivalence of Cantor dynamical systems and study their general consequences, while we examine the structure of the set of morphisms between two primitive substitutive subshifts, which is nonempty only if they are orbit equivalent (by the results of \cite{GiPuSk95} and \cite{Du00}).

Unlike in Section~\ref{sec:UniformCase}, it is now more convenient to consider one-way subshifts $X \subset A^\N$, as this makes indexing slightly less taxing. We note that when studying block maps from one-way subshifts to one-way subshifts, we are simultaneously studying the block maps between the corresponding two-way subshifts. Namely, if $f : X \to Y$ is a block map between two-way subshifts, then for large enough $r \in \N$, we have
\[ x_{[0, \infty)} = x'_{[0, \infty)} \implies f(x)_{[r, \infty)} = f(x')_{[r, \infty)}, \]
for all $x \in X$, so that $\sigma^r \circ f$ is well-defined between right tails of points.

\begin{definition}
\label{def:Cocycle}
Let $X, Y \subset A^\N$ be subshifts, and let $f : X \to Y$ be a function. If $s : X \to \N$ satisfies
\[ f(\sigma(x)) = \sigma^{s(x)}(f(x)) \]
for all $x \in X$, then we call $s$ a \emph{cocycle} for $f$.
\end{definition}

\begin{definition}
Let $X, Y \subset A^\N$ be subshifts. A continuous function $f : X \to Y$ is said to be \emph{orbit-preserving} if $f(\mathcal{O}(x)) \subset \mathcal{O}(f(x))$ holds for all $x \in X$. If $f$ is continuous and has a continuous cocycle $s : X \to \N$ with the property
\begin{equation}
\label{eq:Nontrivial}
\forall x \in X: \exists n \in \N: s(\sigma^n(x)) > 0,
\end{equation}
then we say $f$ is a \emph{\dill{} map}\footnotemark[1], and $s$ is a \emph{nice cocycle} for it.
\end{definition}

\footnotetext[1]{The word `\dill{}' comes from the theory of L systems: a \dill{} map corresponds to a DIL system, that is, a Deterministic Lindenmayer system with Interactions. We add an extra `l' since we are interested in the action of these maps on Long (infinite) words.}

Clearly, a continuous map is orbit-preserving if and only if it has a cocycle, so in particular a \dill{} map is orbit-preserving. Note that a continuous cocycle has a finite image: since $X$ is compact, its image in a continuous map is compact as well.

It follows immediately from the definitions that every block map is a \dill{} map, with the cocycle that sends every configuration to $1$. Also, if $\tau : A \to A^*$ is a substitution, and $X \subset A^\N$ is a subshift with the property that every long enough word in $\B(X)$ contains a letter $a \in A$ with $|\tau(a)| > 0$, then the extension $\tau : X \to A^\N$ is also a \dill{} map, with the natural cocycle $s(x) = |\tau(x_0)|$.

\begin{definition}
Let $X, Y \subset A^\N$ be subshifts, and let $\Phi : X \to Y$ be a \dill{} map. A continuous function $\phi : X \to A^*$ satisfying
\begin{equation}
\label{eq:Implementation}
\Phi(x) = \phi(x) \phi(\sigma(x)) \phi(\sigma^2(x)) \cdots
\end{equation}
for all $x \in X$ is called an \emph{implementation of $\Phi$}.
\end{definition}

\begin{lemma}
\label{lem:LocalRule}
If $\Phi : X \to Y$ is a \dill{} map between subshifts $X, Y \subset A^\N$, then it has an implementation.
\end{lemma}

\begin{proof}
Let $s : X \to \N$ be a nice cocycle for $\Phi$. The implementation can be defined by $\phi(x) = \Phi(x)_{[0,s(x)-1]}$ for all $x \in X$.
\qed
\end{proof}

We usually write ${^n}\phi(x) = \phi(x) \phi(\sigma(x)) \cdots \phi(\sigma^{n-1}(x))$, and (abusing notation), we then have
\[ \Phi(x) = \lim_{n \rightarrow \infty} {^n}\phi(x) = {^\infty} \phi(x). \]

\begin{remark}
\label{rem:NoPeriodicPoints}
If $\Phi : X \to Y$ is a \dill{} map, and $Y$ contains no $\sigma$-periodic points, then the cocycle -- and thus also the implementation -- of $\Phi$ is \emph{unique}.
\end{remark}

In Section~\ref{sec:UniformCase}, we showed that the radius of the local rule of a block map becomes small as it is conjugated by a uniform primitive substitution. Similarly, to prove the main theorem, we will show that when a \dill{} map is conjugated by a substitution, the implementation given by Lemma~\ref{lem:LocalRule} again becomes `small'. To formalize this, we define an analogue of radius for \dill{} maps: the radius pair.

\begin{definition}
Let $\Phi : X \to Y$ be a \dill{} map, and let $r, m \in \N$. If $\Phi$ has an implementation $\phi$ such that
\[ r = \min \{ R \in \N \;|\; x_{[0, R]} = y_{[0, R]} \implies \phi(x) = \phi(y) \} \]
and $m = \max \{ |\phi(x)| \;|\; x \in X \}$,
then $(r,m)$ is a \emph{radius pair of $\Phi$}.
\end{definition}

In the above situation, we write $I(\phi) = r$ and $O(\phi) = m$, and say $r$ and $m$ are the \emph{in-radius} and \emph{out-radius} of $\phi$, respectively. We also note that a \dill{} map can have multiple radius pairs, while Lemma~\ref{lem:LocalRule} implies that it always has at least one: since an implementation is a continuous function from a subshift to a finite set, it must have a finite radius and a finite image.

Finally, we can also compose two \dill{} maps by composing their implementations in the obvious way.

\begin{lemma}
\label{lem:LocalComposition}
Suppose $\Phi_1 : X \to Y$ and $\Phi_2 : Y \to Z$ are \dill{} maps with implementations $\phi_1$ and $\phi_2$, respectively. Then $\Phi_2 \circ \Phi_1 : X \to Z$ is also a \dill{} map, and
\begin{equation*}
\phi(x) = {^{|\phi_1(x)|} \phi_2(^{\infty}\phi_1(x))}
\end{equation*}
is an implementation for it.
\end{lemma}

We also remark that in the above situation,
\begin{equation}
\label{eq:CompoManyFormula}
(^n \phi)(x) = {^{|(^n \phi_1)(x)|} \phi_2(^{\infty}\phi_1(x))}
\end{equation}
holds for all $x \in X$ and $n \in \N$.

\section{Invariants of \Dill{} Maps and their Behavior in Composition}
\label{sec:Invariants}

In this section, we define two invariants, called $Z$ and $D$, that capture how much a word grows when a dill map is applied to it. For an implementation $\phi : X \to A^*$, the invariant $Z(\phi)$ is the limit of $\frac{|(^n\phi)(x)|}{n}$ when this exists and is independent of $x \in X$, and thus measures the size of images of $^n \phi$ in the limit. The invariant $D(\phi)$ measures the maximum of the absolute differences between $|(^n\phi)(x)|$ and $n \cdot Z(\phi)$ for $x \in X$, so that $Z$ and $D$ together give strong analytical information about the \dill{} map. They also give an upper bound for the out-radius $O(\phi)$ when we take $n = 1$. We also discuss the in-radius $I(\phi)$. The inequalities we obtain for the invariants in this section -- in particular Lemma~\ref{lem:Threes} -- are the main technical tool of the proof of the main theorem. While even the existence of $Z$ (and especially $D$) is quite a strong assumption, they exist for both block maps and primitive Pisot substitutions. The claim is trivial for block maps, and it is proved for substitutions in Section~\ref{sec:BalanceAndInvertibility}.

For the rest of this section, we only consider \dill{} maps to which Remark~\ref{rem:NoPeriodicPoints} applies, and for each invariant $H$ we write $H(\Phi)$ for $H(\phi)$, where $\phi$ is the unique implementation of a \dill{} map $\Phi$.

\begin{definition}
Let $\phi : X \to A^*$ be the implementation of a \dill{} map $\Phi : X \to Y$. We define $Z(\phi) = \lambda \in \R$, if $\lim_n \frac{|(^n\phi)(x)|}{n} = \lambda$ for all $x \in X$.
\end{definition}

The invariant $D(\phi)$ is defined in terms of $Z(\phi)$ as follows:

\begin{definition}
Let $\phi : X \to A^*$ be the implementation of a \dill{} map $\Phi : X \to Y$ such that $Z(\phi) = \lambda$. We define $D(\phi)$ as the smallest $D \in \N$ such that for all $x \in X$ and for all $n \in \N$, we have $|(^n \phi)(x)| \in [\lambda n - D, \lambda n + D]$ (when such a $D$ exists).
\end{definition}

In what follows, the notations $Z(\Phi)$ and $D(\Phi)$ imply that the respective invariants are well-defined, unless otherwise noted.

\begin{lemma}
\label{lem:DetermineTheFirst}
Let $\phi : X \to A^*$ be the implementation of a \dill{} map $\Phi : X \to Y$. Then for all $x \in X$ and $m \in \N$, the first $Z(\Phi)^{-1}(m + D(\Phi)) + I(\Phi)$ coordinates of $x$ determine the first $m$ coordinates of $\Phi(x)$.
\end{lemma}

\begin{proof}
We prove that the first $n$ coordinates of $x$ determine (at least) the first $Z(\Phi) (n - I(\Phi)) - D(\Phi)$ coordinates of $\Phi(x)$, for all $n \geq I(\Phi) + 1$: The in-radius of $\phi$ is $I(\Phi)$, so the first $n$ coordinates of $x$ determine that $\Phi(x)$ begins with ${^{n-I(\Phi)}}\phi(x)$. By the definition of $Z(\Phi)$ and $D(\Phi)$, the length of this word is at least $Z(\Phi) (n-I(\Phi)) - D(\Phi)$. The original claim now follows by setting $n = Z(\Phi)^{-1}(m + D(\Phi)) + I(\Phi)$. \qed
\end{proof}

The following lemma quantifies the behavior of the invariants under composition. The invariant $Z$ is quite well-behaved, but the formulas for $D$ and $I$ are less elegant.

\begin{lemma}
\label{lem:InvariantsInComposition}
Let $\Phi_1 : X \to Y$ and $\Phi_2 : Y \to Z$ be \dill{} maps for which the invariants $Z$ and $D$ are well-defined. Then we have
\begin{equation}
\label{eq:CompZ}
Z(\Phi_2 \circ \Phi_1) = Z(\Phi_1) Z(\Phi_2)
\end{equation}
for the invariant $Z$,
\begin{equation}
\label{eq:CompD}
D(\Phi_2 \circ \Phi_1) \leq Z(\Phi_2) D(\Phi_1) + D(\Phi_2)
\end{equation}
for $D$, and
\begin{equation}
\label{eq:CompI}
I(\Phi_2 \circ \Phi_1) \leq Z(\Phi_1)^{-1}(2D(\Phi_1) + I(\Phi_2)) + I(\Phi_1) + 1
\end{equation}
for $I$.
\end{lemma}

\begin{proof}
For clarity, we write $H_i = H(\Phi_i)$ for each invariant $H$ and $i \in \{1, 2\}$. Let also $\phi_i$ be the implementation of $\Phi_i$. Let $x \in X$ be arbitrary. Using \eqref{eq:CompoManyFormula}, we have
\begin{align*}
|(^n \phi)(x)| &= |(^{|(^n \phi_1)(x)|}\phi_2)((^\infty\phi_1)(x))| \\
&= |(^{(1 + \epsilon_1(n, x)) Z_1 n}\phi_2)((^\infty\phi_1)(x))| \\
&= (1 + \epsilon_1(n, x))(1 + \epsilon_2(n, x)) Z_1 Z_2 n,
\end{align*}
where $\epsilon_1(n, x), \epsilon_2(n, x) \rightarrow 0$ as $n \rightarrow \infty$. This proves \eqref{eq:CompZ}. We also have
\begin{align*}
|(^n \phi)(x)| &= |(^{|(^n \phi_1)(x)|}\phi_2)((^\infty\phi_1)(x))| \\
&= |(^{Z_1 n + k_1}\phi_2)((^\infty\phi_1)(x))| & k_1 \in [-D_1, D_1] \\
&= Z_2 Z_1 n + Z_2 k_1 + k_2 & k_2 \in [-D_2, D_2]
\end{align*}
where the $k_i$ are taken among reals, which proves \eqref{eq:CompD}.

It remains to show \eqref{eq:CompI}. For this, we need to bound the in-radius of the function
\[ \phi(x) = {^{|\phi_1(x)|} \phi_2(^{\infty}\phi_1(x))}, \]
that is, we need to bound the number of coordinates of $x \in X$ we need to know in order to determine this word. To compute $\phi(x)$, it is enough to know the value of $|\phi_1(x)|$ and the first $|\phi_1(x)| + I_2$ coordinates of $^{\infty}\phi_1(x)$. The value $|\phi_1(x)|$ is determined by the first $I_1 + 1$ coordinates of $x$. By Lemma~\ref{lem:DetermineTheFirst}, the first $Z_1^{-1}(m + D_1) + I_1$ coordinates of $x$ uniquely determine the first $m$ coordinates of $^{\infty}\phi_1(x)$. Setting $m = O_1 + I_2$, we thus have that the first
\[ k = Z_1^{-1}(O_1 + I_2 + D_1) + I_1 \]
coordinates determine at least the first $O_1 + I_2 \geq |\phi_1(x)| + I_2$ coordinates of $^{\infty}\phi_1(x)$. Since $k > I_1$, the length $|\phi_1(x)|$ is also determined. Finally, $O_1 \leq Z_1 + D_1$ by taking $n = 1$ in the definition of $D_1$. Substituting this into the expression for $k$ gives the formula \eqref{eq:CompI}.
\qed
\end{proof}

In fact, we need such formulas for compositions of three \dill{} maps, the second of which has a $Z$-value of exactly $1$. Proving these inequalities is of course a matter of applying Lemma~\ref{lem:InvariantsInComposition} twice.

\begin{lemma}
\label{lem:Threes}
For some function $C$ from pairs of dill maps to real numbers, assuming $Z_2 = 1$, we have
\begin{align*}
D(\Phi_3 \circ \Phi_2 \circ \Phi_1) &\leq Z_3 (D_1 + D_2) + D_3 \\
&\leq Z_3 D_2 + C(\Phi_1, \Phi_3) \\
I(\Phi_3 \circ \Phi_2 \circ \Phi_1) &\leq Z_1^{-1}(2D_1 + 2D_2 + I_3 + I_2 + 1) + I_1 + 1 \\
&\leq Z_1^{-1} (I_2 + 2 D_2) + C(\Phi_1, \Phi_3)
\end{align*}
\end{lemma}

In particular, if we also have $Z_3 < 1$ and $Z_1 > 1$ (which will be the case in our application), then the contribution of $D_2$ and $I_2$ to the corresponding invariants of $\Phi_3 \circ \Phi_2 \circ \Phi_1$ is of the type in Lemma~\ref{lem:SublinearContribution}.

\section{Balanced Growth and Almost Invertibility}
\label{sec:BalanceAndInvertibility}

In this section, we define notions of balanced growth, almost equivalence and almost invertibility for \dill{} maps. A \dill{} map has balanced growth if all words `blow up' by roughly the same amount in the application of its implementation. Almost equivalence of two dill maps $\Phi_1$ and $\Phi_2$ means that the $\Phi_1$-image and the $\Phi_2$-image of each point differ only by a shift. By almost invertibility, we mean invertibility up to this almost equivalence. For substitutions, we show that almost invertibility follows from primitivity, and balance from both uniformity and the Pisot property. In the following sections, we use this knowledge to generalize the arguments of Section~\ref{sec:UniformCase}.

\begin{definition}
\label{def:Balance}
A \dill{} map has \emph{balanced growth} if the invariants $Z$ and $D$ are well-defined for it.
\end{definition}

This is a nonstandard term. In particular, the concept of balanced growth is not equivalent to the characterization of Sturmian words as aperiodic sequences whose languages are `balanced'. 

\begin{lemma}
\label{lem:ZExists}
For every primitive substitution $\tau : X_\tau \to X_\tau$ seen as a \dill{} map, the invariant $Z(\tau)$ is well-defined and equal to the dominant eigenvalue of $\tau$.
\end{lemma}

\begin{proof}
Since $X_\tau$ is uniquely ergodic \cite{Qu87}, it is easy to see that $Z(\tau)$ is well-defined. The proof of equality of $Z(\tau)$ and the dominant eigenvalue is standard, and we omit it. \qed
\end{proof}

Note that while Lemma~\ref{lem:InvariantsInComposition} shows that $Z$ behaves nicely in the composition of \dill{} maps, we cannot conclude from it and Lemma~\ref{lem:ZExists} that the dominant eigenvalue behaves nicely in the composition of two primitive substitutions. This is because the well-definedness of $Z$ for a substitution $\tau$ depends on the subshift $X_\tau$ more than the action of $\tau$ itself, and if $\rho$ is another substitution, the map $\tau \circ \rho : X_{\tau \circ \rho} \to X_{\tau \circ \rho}$ is not the composition of $\tau : X_\tau \to X_\tau$ and $\rho : X_\rho \to X_\rho$ since the domains and ranges are not compatible. We do not iterate $\tau$ in the definition of $Z(\tau)$, but only in the definition of $X_\tau$, and this iteration is what connects $Z(\tau)$ to the dominant eigenvalue.

It is obvious that uniform substitutions have balanced growth as \dill{} maps, and we next show that this holds also for Pisot substitutions. It is well-known that no substitution is both uniform and Pisot, so that \dill{} maps with balanced growth are a common generalization of the two.

\begin{lemma}[Part of Corollary 2 in \cite{Ad04}]
\label{lem:PisotTheorem}
Let $x \in A^\N$ be a fixed point of a primitive Pisot substitution, and define $\mu(a) = \lim_N \frac{|x_{[0, N-1]}|_a}{N}$ for all letters $a \in A$. Then there exists $C > 0$ such that for all $a \in A$ and $N \in \N$, we have $\big| |x_{[0, N-1]}|_a - N\mu(a) \big| < C$.
\end{lemma}

Actually, Corollary 2 of \cite{Ad04} completely characterizes the substitutions for which the conclusion holds, and in particular shows that they form a larger class than Pisot substitutions. However, the condition is quite involved, so we do not state it here. The balanced growth of primitive Pisot substitutions is an easy corollary of this lemma.

\begin{lemma}
\label{lem:Balanced}
Every primitive Pisot substitution $\tau : X_\tau \to X_\tau$ has balanced growth.
\end{lemma}

\begin{proof}
By Lemma~\ref{lem:ZExists}, $Z(\tau)$ is well-defined. By Lemma~\ref{lem:PisotTheorem}, apart from a uniformly bounded error, $x_{[0, n-1]}$ and $y_{[0, n-1]}$ contain the same number of each letter for all $x, y \in X_\tau$ and $n \in \N$. It is then easy to see that for some $C > 0$, we have
\[ \big| |(^n \tau)(x)| - |(^n \tau)(y)| \big| = \big| |\tau(x_{[0, n-1]})| - |\tau(y_{[0, n-1]})| \big| < C \]
for all $n \in \N$ and $x, y \in X_\tau$. Since the limit of $\frac{|(^n \tau)(x)|}{n}$ is $Z(\tau)$, it is easy to see that for each $n$, there exist points $x, y$ such that $|(^n \tau)(x)| \leq Z(\tau)n \leq |(^n \tau)(y)|$. From this, the claim follows. \qed
\end{proof}

Not all primitive substitutions have balanced growth, even in the binary case, as shown by the following example, which uses the results of \cite{Ad04}.

\begin{example}
Let $\tau$ be the substitution on the binary alphabet defined by $\tau(0) = 0001$ and $\tau(1) = 110$. Then $\tau$ is primitive, and its eigenvalues are $\frac{5 \pm \sqrt{5}}{2}$. Both eigenvalues are strictly greater than $1$, and thus \cite[Theorem 1]{Ad04} implies that there exist pairs of words $u, v \sqsubset X_\tau$ of the same length such that the difference $\big| |u|_0 - |v|_0 \big|$ is arbitrarily large. But then the difference of lengths $\big| |\tau(u)| - |\tau(v)| \big|$ is also arbitrarily large, so that $\tau$ does not have balanced growth.
\end{example}

Next, we define the notions of almost equivalence and almost inverse for \dill{} maps.

\begin{definition}
Let $\Phi, \Psi : X \to Y$ be \dill{} maps with $X$ uniformly recurrent. If
\[ \forall x \in X: \exists i, j \in \N: \sigma^i(\Phi(x)) = \sigma^j(\Psi(x)), \]
then we write $\Phi \sim \Psi$, and say $\Phi$ and $\Psi$ are \emph{almost equivalent}. If $\Phi_1 : X \to Y$, $\Phi_2 : Y \to X$ and $\Phi_2 \circ \Phi_1 \sim \ID_X$, then we say $\Phi_2$ is an \emph{almost left inverse} of $\Phi_1$, and it is an \emph{almost right inverse} if $\Phi_1 \circ \Phi_2 \sim \ID_Y$. An almost right and almost left inverse is called an \emph{almost inverse}. A \dill{} map is \emph{almost invertible} if it has an almost inverse.
\end{definition}

To make sure that it is safe to talk about \dill{} maps `up to almost equivalence', we need to check that almost equivalence is an equivalence relation, and behaves well under composition of \dill{} maps. We omit the (easy) proof.

\begin{lemma}
\label{lem:AlmoIsNice}
Almost equivalence is an equivalence relation, that is,
\[ \Phi_1 \sim \Phi_2 \wedge \Phi_2 \sim \Phi_3 \implies \Phi_1 \sim \Phi_3. \]
Almost equivalence is a congruence with respect to composition, that is,
\[ \Phi_1 \sim \Psi_1 \wedge \Phi_2 \sim \Psi_2 \implies \Phi_2 \circ \Phi_1 \sim \Psi_2 \circ \Psi_1. \]
\end{lemma}

Next, we use the recognizability of primitive substitutions to obtain almost inverse \dill{} maps for them.

\begin{lemma}
\label{lem:InvertibleAndBalanced}
A primitive aperiodic substitution $\tau$ on an alphabet $A$ is almost invertible as a \dill{} map on $X_\tau$. If $\tau$ has balanced growth, then it has an almost inverse with balanced growth.
\end{lemma}

\begin{proof}
First, $\tau$ is strictly recognizable by Lemma~\ref{lem:Mosse}. Let $R : X^\leftrightarrow_\tau \to (A \cup \{\#\})^\Z$ be the block map given by Lemma~\ref{lem:RecogToCA}, and let $r \in \N$ be its (two-sided) radius. Then the neighborhood of $\sigma^r \circ R$ contains only nonnegative elements, so $f_R = \sigma^r \circ R$ can be thought of as a block map from $X_\tau$ to $(\{\#\} \cup A)^\N$. Define $\phi : X_\tau \to A^*$ by
\[ \phi(x) = \left\{\begin{array}{ll}
a, & \mbox{if } f_R(x)_0 = a \in A \\
\epsilon, & \mbox{if } f_R(x)_0 = \#
\end{array}\right. \]
It is easy to see that $\phi$ is the implementation of a \dill{} map $\tau^{-1} = \INF \phi$, which is an almost inverse of $\tau$.

By the way we defined $\tau^{-1}$, the composition $\tau \circ \tau^{-1}$ has an implementation $\psi : X_\tau \to A^*$ such that there exists $C \geq 0$ with $|(^n \psi)(x)| \in [n-C, n+C]$ for all $x \in X_\tau$ and $n \in \N$. It is then easy to see that $\tau^{-1}$ has balanced growth if $\tau$ does.
\qed
\end{proof}

Intuitively, the almost inverse \dill{} map $\tau^{-1}$ of a primitive substitution $\tau$ functions as follows on a configuration $x \in X_\tau$. Since $\tau$ is strictly recognizable, there exist $n \in \N$ and $y \in X_\tau$ such that $\sigma^n(x) = \tau(y)$, where $n$ is bounded by some $r \in \N$, and $y$ is unique once $n$ is determined. This means that the tail $\sigma^n(x)$ is divided into the $\tau$-images of letters of $y$. The implementation of $\tau^{-1}$ simply determines whether $x_r$ lies at the beginning of some image $\tau(y_i)$, giving the letter $y_i$ if this is the case, and giving the empty word $\epsilon$ otherwise. By choosing $n$ as close to $r$ as possible, we then have $\tau^{-1}(x) = y$.

As an example, we give an almost inverse of the Thue-Morse substitution.

\begin{example}
\label{ex:ThueMorseInvertible}
Let $\tau$ be the Thue-Morse substitution, and let the implementation $\tau^{-1} : X_\tau \to \{0,1\}^*$ be defined by
\[ \begin{array}{cc}
\tau^{-1}(00101 \ldots) = \epsilon,	& \tau^{-1}(00110 \ldots) = \epsilon \\
\tau^{-1}(01001 \ldots) = \epsilon,	& \tau^{-1}(01011 \ldots) = 0 \\
\tau^{-1}(01100 \ldots) = 0,		& \tau^{-1}(01101 \ldots) = 0 \\
\tau^{-1}(10010 \ldots) = 1,		& \tau^{-1}(10011 \ldots) = 1 \\
\tau^{-1}(10100 \ldots) = 1,		& \tau^{-1}(10110 \ldots) = \epsilon \\
\tau^{-1}(11001 \ldots) = \epsilon,	& \tau^{-1}(11010 \ldots) = \epsilon \\
\end{array} \]
One can check that $\tau \circ (^\infty \tau^{-1}) \sim \ID_{X_\tau} = (^\infty \tau^{-1}) \circ \tau$, so that ${}^\infty \tau^{-1}$ is an almost inverse for $\tau$. \qed
\end{example}


Finally, we characterize the almost equivalence of block maps on our subshifts of interest.

\begin{lemma}
\label{lem:Almo}
Let $f, g : X \to Y$ be almost equivalent block maps, where $X$ is uniformly recurrent. Then $g \circ \sigma^k = f$ or $g = f \circ \sigma^k$ for some $k \in \N$.
\end{lemma}

\begin{proof}
By symmetry and the fact that block maps commute with shifts, we may choose $x \in X$ such that $f(x) = \sigma^k(g(x))$ for some $k \in \N$. Denote $h = \sigma^k \circ g$. Let $y \in X$ be arbitrary, and let $r \in \N$ be a common radius for $f$ and $h$. Since $X$ is uniformly recurrent, there exists $i \in \N$ such that $y_{[0, r-1]} = x_{[i, i+r-1]}$. But then
\begin{align*}
f(y)_0 &= f(y_{[0, r-1]}) = f(x_{[i, i+r-1]}) = f(x)_i \\
&= h(x)_i =h (x_{[i, i+r-1]}) = h(y_{[0, r-1]}) = h(y)_0,
\end{align*}
which implies $f(y) = h(y)$, since $h$ commutes with $\sigma$. Since $y \in X$ was arbitrary, we have $f = h$. \qed
\end{proof}

In particular, the above lemma holds when $X$ and $Y$ are subshifts of some primitive substitutions by Lemma~\ref{lem:MossePower}.

Lemma~\ref{lem:Almo} does not apply in the case that $f \sim \Phi$, where $f$ is a block map and $\Phi$ a general \dill{} map. For example, on the Thue-Morse shift, $\Phi$ could behave differently depending on whether the input point is of the form $\tau(x)$ or $\sigma(\tau(x))$ for some $x \in X_\tau$.

\section{Invariants of \Dill{} Maps between Substitutions}
\label{sec:Conjugation}

In this section, we generalize the conjugation argument of Proposition~\ref{prop:UniformCase} to all primitive substitutions with balanced growth. For the almost inverse \dill{} map $\tau^{-1}$ of a substitution $\tau$ and $j \in \N$, we denote $\tau^{-j} = (\tau^{-1})^j$.

\begin{example}
We continue Example~\ref{ex:ThueMorseInvertible}. The endomorphisms of the Thue-Morse subshift $X_\tau$ are characterized in \cite{Co71}, and they are the shift maps, possibly composed with a bit flip. We show that the process of repeated conjugation eventually sends each such map into a finite set of maps. First, we consider the shift maps and show that
\[ \tau^{-1} \circ \sigma^n \circ \tau = \sigma^{\lceil n/2 \rceil} \]
for all $n \in \N$. Then we have $\tau^{-j} \circ \sigma^n \circ \tau^j \in \{\ID_{X_\tau}, \sigma\}$ for large enough $j \in \N$. The odd case $n = 2k + 1$ is the more interesting one, and we have
\begin{align*}
x = x_0 x_1 x_2 x_3 \cdots &\stackrel{\tau}{\mapsto} x_0 \overline{x_0} x_1 \overline{x_1} x_2 \overline{x_2} x_3 \overline{x_3} \cdots \\
&\stackrel{\sigma^n}{\mapsto} \overline{x_k} x_{k+1} \overline{x_{k+1}} x_{k+2} \overline{x_{k+2}} x_{k+3} \overline{x_{k+3}} \cdots \\
&\stackrel{\tau^{-1}}{\mapsto} \epsilon \cdot x_{k+1} \cdot \epsilon \cdot x_{k+2} \cdot \epsilon \cdot x_{k+3} \cdot \epsilon \cdots \\
&= \sigma^{k+1}(x).
\end{align*}
As for the bit flip $g : X_\tau \to X_\tau$, a similar computation shows that
\[ \tau^{-1} \circ \sigma^n \circ g \circ \tau = \sigma^{\lceil n/2 \rceil} \circ g, \]
so that for any endomorphism $f$ of the Thue-Morse subshift $X_\tau$, we have
\[ \tau^{-j} \circ f \circ \tau^j \in \{\ID_{X_\tau}, g, \sigma, \sigma \circ g\}\]
for large enough $j \in \N$.
\qed
\end{example}

\emph{For the rest of this section, suppose $\tau$ and $\rho$ are primitive substitutions with balanced growth and $\lambda = Z(\tau) = Z(\rho) > 1$.}

By Lemma~\ref{lem:InvertibleAndBalanced}, there exists an almost inverse $\rho^{-1}$ for $\rho$ that also has balanced growth. Let $f : X_\tau \to X_\rho$ be a block map, which we consider as a \dill{} map. We will repeatedly `conjugate' $f$ with $\tau$ and $\rho$, obtaining a sequence $(\Phi_i)_{i \in \N}$ with $\Phi_0 = f$ and $\Phi_{i+1} = \rho^{-1} \circ \Phi_i \circ \tau$, and observe the invariants $Z$, $D$, and $I$ of the $\Phi_i$. Thinking of $\tau$ and $\rho$ as fixed, we show the following evolution for the invariants:
\begin{itemize}
\item $Z(\Phi_i)$ stays at $1$,
\item $D(\Phi_i)$ stays bounded by a constant,
\item $I(\Phi_i)$ decreases until it is bounded by a constant.
\end{itemize}
This is the content of the following lemma.

\begin{lemma}
\label{lem:Conjugation}
There exist $D = D_{\tau, \rho} \in \N$ and $I = I_{\tau, \rho} \in \N$ with the following properties. For any block map $f : X_\tau \to X_\rho$, denoting $\Phi_i = \rho^{-i} \circ f \circ \tau^i$, we have $Z(\Phi_i) = 1$ and $D(\Phi_i) \leq D$ for all $i \in \N$, and $I(\Phi_i) \leq I$ for all large enough $i \in \N$.
\end{lemma}

\begin{proof}
The first claim follows directly from equation~\eqref{eq:CompZ} of Lemma~\ref{lem:InvariantsInComposition} and the observation $Z(f) = 1$. For the second claim, we first note that $D(\Phi_0) = 0$, since $\Phi_0 = f$ is a block map. From Lemma~\ref{lem:Threes}, we then obtain
\[ D(\Phi_{i+1}) \leq \lambda^{-1} D(\Phi_i) + \lambda^{-1} D(\tau) + D(\rho^{-1}), \]
so the claim follows from Lemma~\ref{lem:SublinearContribution}. Finally, Lemma~\ref{lem:Threes} also implies that
\[
I(\Phi_{i+1}) \leq \lambda^{-1} I(\Phi_i) + 2 \lambda^{-1} D(\Phi_i) + C,
\]
where $C > 0$ only depends on $\tau$ and $\rho$. Since we have already proved that $D(\Phi_i)$ is bounded, the final claim follows from Lemma~\ref{lem:SublinearContribution}.
\qed
\end{proof}

We now put the invariants into use and show that when they are all bounded, there are only finitely many choices for the corresponding \dill{} map.

\begin{lemma}
\label{lem:SmallImplementation}
Let $X, Y \subset A^\N$ be subshifts, of which $Y$ is aperiodic. For any $Z, D, I \in \N$, there are only finitely many \dill{} maps $\Phi : X \to Y$ with
\begin{equation}
\label{eq:SmallInvariants}
Z(\Phi) \leq Z, D(\Phi) \leq D, I(\Phi) \leq I.
\end{equation}
\end{lemma}

\begin{proof}
If $\Phi : X \to Y$ is such a \dill{} map, its unique implementation has in-radius at most $I$ and out-radius at most $Z + D$. The claim follows because there are at most $(\sum_{k=0}^{Z + D} |A|^k)^M$ choices for such a function, where $M = |A|^{I+1}$. \qed
\end{proof}

\section{Block Maps between Substitutions of Balanced Growth}
\label{sec:Finite}

We are now ready to prove our main theorem. A similar (but stronger) result was proved in \cite{HoPa89} for uniform primitive substitutions. In \cite{Ol13}, it was proved for a general class of Sturmian systems (including for example the subshift generated by the Fibonacci substitution) that the endomorphism monoid consists of the shift maps only.

\begin{theorem}
\label{thm:FinitelyManyCA}
Let $\tau$ and $\rho$ be primitive aperiodic substitutions with balanced growth whose associated matrices have the same dominant eigenvalue $\lambda > 1$. Then there exists a finite set $P$ of block maps from $X_\tau$ to $X_\rho$ such that if $f : X_\tau \to X_\rho$ is a block map, then $f = \sigma^k \circ g$ or $g = \sigma^k \circ f$ for some $g \in P$ and $k \in \N$.
\end{theorem}

In terms of our invariants, the assumptions are that $Z(\tau), Z(\rho), D(\tau)$ and $D(\rho)$ all exist and $Z(\tau) = Z(\rho) = \lambda > 1$. Both uniform and Pisot substitutions were shown to have balanced growth in Section~\ref{sec:BalanceAndInvertibility}, so the result holds in particular for these classes. It is known that $Z(\tau) = Z(\rho)$ is impossible if one of $\tau$ and $\rho$ is uniform and the other is Pisot.


\begin{proof}
Let $Q$ be the set of \dill{} maps from $X_\tau$ to $X_\rho$ satisfying \eqref{eq:SmallInvariants} for $Z = 1$, $D = D_{\tau, \rho}$ and $I = I_{\tau, \rho}$. The set $Q$ is finite by Lemma~\ref{lem:SmallImplementation}.

Let $f : X_\tau \to X_\rho$ be a block map, which is in particular a \dill{} map. Since $\tau$ and $\rho$ are balanced, and almost invertible by Lemma~\ref{lem:InvertibleAndBalanced}, we can apply Lemma~\ref{lem:Conjugation} to see that for some $n \in \N$, the invariants $D(\Phi_i)$ and $I(\Phi_i)$ of $\Phi_i = \rho^{-i} \circ f \circ \tau^i$ are smaller than the respective constants $D_{\tau, \rho}$ and $I_{\tau, \rho}$ for all $i \geq n$. Then $\Phi_i \in Q$ for all $i \geq n$.

Since $Q$ is finite, we must in fact have $\Phi_{i + j} = \Phi_i$ for some $i \geq n$ and $j > 0$. Because the operation $\Phi \mapsto \rho^{-1} \circ \Phi \circ \tau$ is bijective up to almost equivalence (its reverse being $\Phi \mapsto \rho \circ \Phi \circ \tau^{-1}$), we must then have $f \sim \Phi_{i+\ell} \in Q$ for some $\ell \in [0,j-1]$ by Lemma~\ref{lem:AlmoIsNice}.

If two block maps are almost equivalent to the same \dill{} map $\Phi \in Q$, they are almost equivalent to each other by Lemma~\ref{lem:AlmoIsNice}, and thus one is a shift of the other by Lemma~\ref{lem:Almo}. We choose $P$ as a set of representatives of these finitely many almost equivalence classes. \qed
\end{proof}

As we remarked in Section~\ref{sec:OrbitPreservingMaps}, Theorem~\ref{thm:FinitelyManyCA} can be applied also in the case of two-sided subshifts, as each block map between the two-sided subshifts becomes a block map between the corresponding one-sided subshifts when composed with a large enough power of the shift. Almost equivalence is also preserved by this operation (with the obvious definition in the two-sided case). This gives the following direct corollary.

\begin{theorem}
\label{thm:FinitelyManyCA2way}
Let $\tau$ and $\rho$ be primitive aperiodic substitutions with balanced growth whose associated matrices have the same dominant eigenvalue $\lambda > 1$. Then there exists a finite set of block maps $P$ from $X_\tau^\leftrightarrow$ to $X_\rho^\leftrightarrow$ such that if $f : X_\tau^\leftrightarrow \to X_\rho^\leftrightarrow$ is a block map, then $f = \sigma^k \circ g$ for some $g \in P$ and $k \in \Z$.
\end{theorem}

By the results of \cite{Du00}, all endomorphisms of a two-sided subshift generated by a primitive substitution are in fact automorphisms: since the subshift generated is minimal, all endomorphisms are factor maps, and \cite[Corollary~18]{Du00} then shows that they are automorphisms. In the following, if a group $G$ has a subgroup of finite index isomorphic to another group $H$, then we say that $G$ is \emph{virtually $H$}. Using these observations, we can state Theorem~\ref{thm:FinitelyManyCA2way} in a more simple form in the special case $\tau = \rho$

\begin{corollary}
\label{cor:FinitelyManyCA2way2}
The endomorphism monoid of the two-sided subshift of a primitive aperiodic substitution of balanced growth is a group, and is virtually $\Z$ (the subgroup isomorphic to $\Z$ being the subgroup of shifts).
\end{corollary}

Note that it is not true for all substitutions that all endomorphisms are surjective, as the subshift generated by the substitution
\[ 0 \mapsto 010, 1 \mapsto 11 \]
has the non-surjective endomorphism $x \mapsto {}^\infty 1^\infty$. In particular, the assumption of primitivity is needed in Corollary~\ref{cor:FinitelyManyCA2way2}.

Recall Example~\ref{ex:ManyUniform} from the uniform case, where we constructed an arbitrary number of not almost equivalent endomorphisms with arbitrarily large radii on the subshift of a primitive uniform substitution. One might ask whether such a construction can be carried out in the non-uniform case. The following shows the answer to be positive, with surprisingly few modifications.


\begin{example}
Recall the notation of Example~\ref{ex:ManyUniform}. Let again $m \in \N$ and $n \geq 4$, and define $\tau(a_i) = b_{i+1} a_i^{n-1}$ and $\tau(b_i) = b_i a_i^n$. Now the substitution $\tau$ is not uniform, but the original argument can be directly applied, providing a non-uniform primitive substitution and $m-1$ pairwise ${\sim}$-nonequivalent endomorphisms with large neighborhoods. \qed
\end{example}

The substitution in the previous example is not Pisot, and we do not know examples of Pisot substitutions with large numbers of nonequivalent endomorphisms. For example, \cite{Ol13} shows that for Sturmian substitutions (which are a subclass of Pisot substitutions), the only morphisms are the shift maps.

While we strongly depend on the Pisot property (or rather, balanced growth) in our argument, we believe it is not needed for the result that the automorphism group is virtually $\Z$. One can also ask if, more generally, the automorphism group of a linearly recurrent subshift is virtually $\Z$. We would not be particularly surprised if this were the case, but it does not seem likely to us that our method can tackle this problem.

\begin{conjecture}
The automorphism group of the subshift of a primitive aperiodic substitution is virtually $\Z$.
\end{conjecture}

\begin{question}
Is the automorphism group of every linearly recurrent subshift virtually $\Z$?
\end{question}

Our result could also be generalized in another direction, namely, measure-preserving maps between subshifts generated by primitive substitutions. Note that the subshift of a primitive substitution supports a unique shift-invariant probability measure \cite{Du00}. The uniform case has been partially solved in \cite{HoPa89}, where it was shown that there are only finitely many (up to shifting and almost everywhere equivalence) measure-preserving, almost everywhere shift-commuting maps between two subshifts generated by uniform primitive substitutions satisfying certain injectivity conditions.

\begin{question}
Does the result of \cite{HoPa89} hold in the Pisot case? The general primitive case?
\end{question}


\section*{Acknowledgements}

We are thankful to our advisor Jarkko Kari and Luca Zamboni for their guidance and support, Fabien Durand for his help with finding references, Timo Jolivet for suggesting the idea behind Example~\ref{ex:ManyUniform}, Charalampos Zinoviadis for his help on improving the style of this article, and of course the anonymous referee for their valuable comments.

\bibliographystyle{plain}
\bibliography{../../../bib/bib}{}

\end{document}